\documentclass[preprint,12pt]{elsarticle}
\usepackage{amssymb}
\usepackage{mathrsfs}
\usepackage{amsmath,amstext,amsfonts,verbatim}

\usepackage[latin1]{inputenc}

\usepackage{amsmath,amsthm,amssymb}

\usepackage{amsfonts}

\usepackage{amsmath,amssymb}
\usepackage{graphicx}
\usepackage{color}
\usepackage{indentfirst}

\journal{arXiv}

\newtheorem{theorem}{Theorem}[section]
\newtheorem{lemma}[theorem]{Lemma}

\theoremstyle{definition}
\newtheorem{definition}[theorem]{Definition}

\newtheorem{prop}{Proposition}[section]

\theoremstyle{remark}

\numberwithin{equation}{section}

\newcommand{\neweq}[1]{\begin{equation}\label{#1}}

\def\phi{\varphi}

\def\incep{\left\{\begin{array}{cl} }
 \def\termin{\end{array}\right. }
\def\2af{2^*_\alpha}

\begin{document}

\begin{frontmatter}

\title{The Shannon-McMillan-Breiman theorem of random dynamical systems for amenable group actions\tnoteref{t1}}
\tnotetext[t1]{This document is the results of the research project funded by Fundamental Research Program of Shanxi Province (No.20210302123322), the Science and Technology Research Program of Chongqing Municipal Education
Commission (Grant No.KJQN202400826) and Research Project of Chongqing Technology and Business University(No. 2256015).}

\author{Yuan Lian}
\address{College of Mathematics and Statistics, Taiyuan Normal University, Taiyuan 030619, China}
\ead{andrea@tynu.edu.cn}

\author{Bin Zhu}
\address{School \ of \ Mathematics \ and \ Statistics\ , Chongqing Technology and Business University, Chongqing, 400067, China}
\ead{binzhucqu@163.com}




\begin{abstract}

The Shannon-McMillan-Breiman theorem is one of the most important results in information theory, which can describe the random ergodic process, and its proof uses the famous Birkhoff ergodic theorem, so it can be seen that it plays a crucial role in ergodic theory. In this paper, the Shannon-McMillan-Breiman theorem in the random dynamical systems is proved from the perspective of an amenable group action, which provides a boost for the development of entropy theory in the random dynamical systems for amenable group actions.

\end{abstract}





\end{frontmatter}

\section{Introduction}

Entropy is one of the fundamental invariants widely used to describe the complexity of dynamical systems, and metric theory entropy measures the average amount of information and complexity of a system. As we all know, the measure entropy of a system describes the information creation rate of evolution, while the topological entropy measures the total exponential complexity of the orbital structure.

In the study of deterministic dynamical systems, Katok \cite{K} introduces the equivalence definition of measure entropy in a method similar to the definition of topological entropy.

A fundamental pillar of information theory is the Shannon-McMillan-Breiman (SMB) theorem. This theorem was first proved by\ Claude and\ Shannon in the famous paper \ \cite{CS} in 1948, and later extended to the discrete stochastic process \ \cite{MB} by \ McMillan in 1953, and the almost-sure convergence was proved by \ Breiman \ \cite{BL} in 1957 and the case of countable partitions was proved by \ Chung \cite{C}.

Amenable groups include finite groups and compact groups, and the action of these groups on the compact metric space is a natural expansion of integer group action. Bogensch\"{u}tz \ \cite{B} introduces the Shannon-McMillan-Breiman theorem in random dynamical systems. Lindenstrauss \ \cite{L}, Ornstein and Weiss \ \cite{OW} proved the Shannon-McMillan-Breiman theorem in the random dynamical system for an amenable group action under certain conditions, while in \ \cite{W} Weiss proved the non-ergodic case in an system for amenable group actions (see also \ \cite{ZC}). For the theories for amenable group actions, please refer to \ \cite{KL}. In this paper, we will consider extending the Shannon-McMillan-Breiman theorem to the random dynamical systems for an amenable group action.

\section{Group actions}

Let $X$ be a compact metric space, $d$ is a measure on $X$, $m$ is a Borel probability measure, and $f:X\rightarrow X$ is a continuous mapping that guarantees Borel probability measure $m$ invariant. $(\Omega, \mathcal{F},\mathbb{P},G)$ is a measurable dynamical system (abbreviated as MDS), with the identity element of $G$ denoted by $e_{G}$. Please refer to \cite{DZ} for the following sections. In 1983, Brin and Katok solved the problem posed by Young and Ledrappier, thus giving a topological version of the Shannon-McMillan-Breiman (SMB) theorem, i.e. the local entropy formula.

\begin{definition}
Suppose $(\Omega,\mathcal{F},\mathbb{P},G)$ is MDS, where $(\Omega,\mathcal{F},\mathbb{P})$ is a  Lebesgue space. In particular, $(\Omega,\mathcal{F},\mathbb{P})$ is complete countable separated.

Now suppose $(X,\mathcal{B})$ is a measurable space and $\mathcal{E}\in\mathcal{F}\times\mathcal{B}$, then $(\mathcal{E}, (\mathcal{F}\times\mathcal{B})_{\mathcal{E}})$ naturally forms a measurable space, where $(\mathcal{F} \times\mathcal{B})_{\mathcal{E}}$ is a $\mathcal{E}$ algebra, $\mathcal{F}\times\mathcal{B}$ limit on $\mathcal{E}$ constitutes that.

For any \ $\omega\in\Omega$, let \ $\mathcal{E}_{\omega}=\{x\in X:(\omega,x)\in \mathcal{E}\}$ be the \ $\omega$ part of \ $\mathcal{E}$.

A random dynamical system (or simply RDS) defined on \ $(X,\mathcal{B})$ related to \ $(\Omega,\mathcal{F},\mathbb{P},G)$ is a family of mappings
$$\mathbf{F}=\{F_{g,\omega}:\mathcal{E}_{\omega}\rightarrow \mathcal{E}_{g\omega}\mid g\in G, \omega\in \Omega\}$$
that satisfies the following conditions:

(1) For each \ $\omega\in\Omega$, the transformation \ $F_{e_{G},\omega}$ is a constant mapping on \ $\mathcal{E}_{\omega}$ ;

(2) For any \ $g\in G$, the mapping
$(\mathcal{E},(\mathcal{F}\times\mathcal{B})_{\mathcal{E}})\rightarrow (X,\mathcal{B})$
defined by
\ $(\omega,x)\mapsto F_{g,\omega}(x)$
is measurable;

(3) for arbitrary \ $\omega\in\Omega$ and \ $g_{1},g_{2}\in G, F_{g_{2},g_{1}\omega}\circ F_{g_{1},\omega}=F_{g_{2}g_{1},\omega}$ (so for any \ $g\in G$, there is \ $F_{g^{-1},\omega}=(F_{g,g^{-1}\omega})^{-1}$).

\end{definition}

In this case, for any \ $g\in G$  to define a mapping \ $\Theta:G\times \mathcal{E}\longrightarrow \mathcal{E}$ by way of \ $(\omega,x)\rightarrow(g\omega,F_{g,\omega}x)$, then \ $G$  naturally forms a measurable action on \ $\mathcal{E}$, which is called the corresponding skew product transformation.

\begin{definition}
Let \ $\mathbf{F}=\{F_{g,\omega}:\mathcal{E}_{\omega}\rightarrow \mathcal{E}_{g\omega}\mid g\in G, \omega\in \Omega\}$  be an RDS defined on \ $(X,\mathcal{B})$ related to \ $(\Omega,\mathcal{F},\mathbb{P},G)$, where \ $X$ is a compact space with a metric \ $d$ and the Borel $\sigma$ algebra is \ $\mathcal{B}$. If for \ $\mathbb{P}-$ a.e. $\omega\in\Omega$, $\emptyset\neq\mathcal{E}_{\omega}\subseteq X$  is a compact subset, and for any \ $g\in G$, $F_{g,\omega}$ is a continuous mapping (so \ $\mathbb{P}-$ a.e. $\omega\in\Omega$ and \ $g\in G$,  $F_{g,\omega}:\mathcal{E}_{\omega}\rightarrow \mathcal{E}_{g\omega}$ are homomorphic maps), then \ $\mathbf{F}$ is said to be defined as a continuous random dynamical system (abbreviated as CRDS) with respect to $(\Omega,\mathcal{F},\mathbb{P},G)$ on \ $(X,\mathcal{B})$.
\end{definition}

In this paper, we use\ $P(X)$ denotes the family of sets consisting of all finite measurable partitions of \ $X$, and $\mathcal{P}_{\mathbb{P}}(\Omega\times X)$ denotes the set of all probability measures $\mu$  on \ $\Omega\times X$ satisfying \ $\mu\circ\pi^{-1}=\mathbb{P}$, where \ $\pi_{\Omega}:\Omega\times X\rightarrow \Omega$ is the natural projection. Remember
$$\mathcal{P}_{\mathbb{P}}(\mathcal{E})=\{\mu\in\mathcal{P}_{\mathbb{P}}(\Omega\times X):\mu(\mathcal{E})=1\}.$$

$\mu\in\mathcal{P}_{\mathbb{P}}(\mathcal{E})$ can be decomposed into \ $d\mu(\omega,x)=d\mu_{\omega}(x)d\mathbb{P}(\omega)$, where \ $\mu_{\omega}\ (\ \omega\in \Omega$ ) is a regular conditional probability measure related to \ $\sigma$ algebra \ $\mathcal{F}_{\mathcal{E}}$. Since \ $X$ is a compact metric space, then \ $\mu$ has a decomposition \cite{DRM}, so that for any measurable set $\mathcal{R}\subseteq \mathcal{E}$, there is
$\mu(\mathcal{R})=\int_{\Omega}\mu_{\omega}\mathcal{R}_{\omega}d\mathbb{P}(\omega)$.

Let \ $\mathcal{P}_{\mathbb{P}}(\mathcal{E},G)$ denote the set of all \ $G$-invariant elements in \ $\mathcal{P}_{\mathbb{P}}(\mathcal{E})$. $\mathcal{E}_{\mathbb{P}}(\mathcal{E},G)$ denotes the set of all ergodic elements in \ $\mathcal{P}_{\mathbb{P}}(\mathcal{E})$. Let \ $C_{\mathcal{E}}$ denote the set of all finite covers of \ $\mathcal{E}$, and $P(\mathcal{E})$ denote the set of all finite partitions of \ $\mathcal{E}$. $\mathcal{F}(G)$ denotes the non-empty finite subsets in \ $G$. Let \ $\mathcal{U}\in C_{\mathcal{E}}$, $F\in \mathcal{F}(G)$, and use \ $\mathcal{U}^{F}$ to denote \ $\bigvee_{g\in F}g^{-1}\mathcal{U}$.

\begin{definition}

Let \ $\mathcal{U}\in C_{\mathcal{E}}$, $\mu\in \mathcal{P}_{\mathbb{P}}(\mathcal{E},G)$, see\ \cite{DZ}, \ $\mu$ fibre entropy of $\mathbf{F}$ with respect to \ $\mathcal{U}$ is
\begin{align*}
h^{(r)}_{\mu}(\mathbf{F},\mathcal{U})=\lim_{n\rightarrow\infty}\frac{1}{|F_{n}|}H_{\mu}(\mathcal{U}_{F_{n}}|\mathcal{F}_{\mathcal{E}})
\end{align*}
and \ $\mu$ fibre entropy of \ $\mathbf{F}$ is
\begin{align*}
h^{(r)}_{\mu}(\mathbf{F})=\sup_{\mathcal{U}\in P(\mathcal{E})}h^{(r)}_{\mu}(\mathbf{F},\mathcal{U}),
\end{align*}
where
$$H_{\mu}(\mathcal{U}|\mathcal{F}_{\mathcal{E}})=\int_{\Omega}H_{\mu_{\omega}}(\mathcal{U}_{\omega})d\mathbb{P}(\omega),
\ \mathcal{U}_{\omega}=\{U_{\omega}:U\in\mathcal{U} \}$$

$$(\mathcal{U}_{F})_{\omega}=\bigvee_{g\in F}(g^{-1}\mathcal{U})_{\omega}=\bigvee_{g\in F}(F_{g,\omega})^{-1}\mathcal{U}_{g\omega}=\bigvee_{g\in F}F_{g^{-1},g\omega}\mathcal{U}_{g\omega}.$$

\end{definition}

Let \ $\pi:(Y_{1},G)\rightarrow (Y_{2},G)$  be a factor mapping between two dynamical systems, then there is
$$h_{\nu_{1}}(G,Y_{1}|\pi)=h_{\nu_{1}}(G,Y_{1}|\pi^{-1}\mathcal{B}_{Y_{2}}).$$

\section{The main result}

Next, we will first give the metric
$$d^{\omega}_{E}(x,y)=\max_{s\in E}d(F_{s,\omega}x,F_{s,\omega}y) \quad \forall x,y\in X, \forall E\in \mathcal{F}(G), \forall \omega\in \Omega.$$
which is defined on $X$.

Let \ $B_{d^{\omega}_{E}}(x,y)$ represent a $d^{\omega}_{E}$ ball with centre \ $x$  and radius $\varepsilon$.

To prove the main results of this article, we first consider a finite measurable partition \ $\xi$ of \ $\Omega\times X$. The following propositions can be regarded as a generalization of Proposition 2.2 in \ \cite{ZY}, while the Proposition 2.2 in \ \cite{ZY} only considers the case of classical random dynamical systems; this article show the corresponding results to the random dynamical systems for amenable group actions.

Let\ $g\in G$ and \ $E\in \mathcal{F}(G)$. For arbitrary finite partition \ $\xi$ of \ $\Omega\times X$ and arbitrary \ $\omega\in\Omega$, let \ $\xi_{\omega}=\{A_{\omega}:A\in \xi\}$. \ $\xi^{E}_{\omega}(x)$ denotes the element of \ $\bigvee_{g\in E}F_{g,\omega}^{-1}\xi_{g\omega}$ containing \ $x$  (specifically, \ $\xi_{\omega}(x)=\xi^{\{e\}}_{\omega}(x)$).

The Proposition \ \ref{prop1} in this paper requires the following lemma \ \ref{lem11} and lemma \ \ref{lem12}, lemma \ \ref{lem11} and lemma \ \ref{lem12} from the lemma 2.1 and lemma 2.3 in \ \cite{L}.

\begin{lemma}\label{lem11}
Let \ $\delta>0$, $\epsilon$  be small and  $M$ is a large number related to \ $\delta$. Let \ $\overline{F}_{1},\overline{F}_{2},...,\overline{F}_{M}\subseteq G$ be a consequence of finite subsets satisfying
$$
\mid\bigcup_{j\leq i}\overline{F}^{-1}_{j}\overline{F}_{i+1}\mid<(1+\epsilon)\mid\overline{F}_{i+1}\mid,
$$
suppose $F$  be another subset of  $G$  (usually larger than  $\overline{F}_{M}$), and for \ $i=1,2,...,M,$ let  \ $A_{i}\subseteq F$  satisfy \ $F_{i}A_{i}\subseteq F$. Then there if \ $B(i,a)$ that is either  \ $\overline{F}_{i}a$ or \ $\varnothing$ for $i=1,2,...,M,a\in A_{i}$ such that
$$
(1+\delta)\mid\bigcup_{i,a}B(i,a)\mid\geq\sum_{i,a}\mid B(i,a)\mid\geq\min_{i}|A_{i}|-\delta|F|.
$$
\end{lemma}

\begin{lemma}\label{lem12}

For any \ $\delta, C>0$  and finite subset \ $K\subseteq G$, if \ $M$ is large enough and $\epsilon$  is small enough (depending on the preceding arguments), the following holds. Suppose given a finite subset $\overline{F}_{i,j}\subseteq G$($1=1,2,...,M,j=1,2,...,N_{i}$) such that
$$
\mid\bigcup_{(i',k')\preceq(i,k)}\overline{F}^{-1}_{i',k'}\overline{F}_{i,k+1}\mid\leq C\mid\overline{F}_{i,k+1}\mid, k=1,2,...,N_{i}-1,
$$
$$
\mid\bigcup_{(i',k')\preceq(i,N_{i})}\overline{F}^{-1}_{i',k'}\overline{F}_{i+1,k}\mid\leq (1+\epsilon)\mid\overline{F}_{i+1,k}\mid, k=1,2,...,N_{i+1},
$$
If $A_{i,j}\subseteq F$($1=1,2,...,M,j=1,2,...,N_{i}$)  such that $\overline{F}_{i,j}A_{i,j}\subset F$ and take $\alpha$ so that for all $i$
$$
|\bigcup_{j=1}^{N_{i}}KA_{i,j}|\geq \alpha|F|.
$$
Then it is possible to find set-valued random variables $B(i,j,a)$ (for $1=1,2,...,M,j=1,2,...,N_{i}$) and $a\in A_{i,j}$ such that

1.$B(i,j,a)$ is either \ $\overline{F}_{i}a$ or \ $\varnothing$;

2.Let \ $\Lambda:F\rightarrow \mathbb{N}$ be a random function \ $\Lambda(g)=\sum_{i,j,a}1_{B(i,j,a)}(g)$, then for any \ $g\in F$,
$$
E(\Lambda(g)\mid\Lambda(g)>0)<(1+\delta);
$$

3.$E(\sum_{g\in G}\Lambda(g))=E(\sum_{i,j,a}\mid B(i,j,a)\mid)>(\alpha-\delta)\mid F \mid$.
\end{lemma}

\begin{prop}\label{prop1}
Let \ $(X,d)$ be a compact metric space, $G$ be an infinite countable discrete amenable group, and $\{F_{n}\}$ be a tempered F{\o}lner sequence of \ $G$, which satisfies \ $e_{G}\in F_{1}\subseteq F_{2}\subseteq F_{3}\subseteq\cdot\cdot\cdot, $ and for any $n\in\mathbb{N}, |F_{n}|>n$.

If \ $\mathbf{F}$ is a CRDS defined on \ $(X,\mathcal{B})$ related to \ $(\Omega,\mathcal{F},\mathbb{P},G)$, $\mu\in \mathcal{P}_{\mathbb{P}}(\mathcal{E},G)$, then for any finite measurable partition \ $\xi$ of $\Omega\times X$, the following can be obtained

(1)In \ $L^{1}(\Omega\times X,\mu)$, $\lim_{n\rightarrow\infty}\frac{1}{|F_{n}|}I_{\mu_{\omega}}(\bigvee_{g\in F_{n}}F_{g,\omega}^{-1}\xi_{g\omega})(x)=E_{\mu}(r|\mathcal{I})(\omega,x),\mu-a.e.$, where
$$I_{\mu_{\omega}}(\bigvee_{g\in F_{n}}F_{g,\omega}^{-1}\xi_{g\omega})(x)=-\log\mu_{\omega}(\xi^{F_{n}}_{\omega}(x)), n\in\mathbb{N}.$$
is an information function, $r(\omega,x)=\lim_{n\rightarrow\infty}I_{\mu_{\omega}}(\xi_{\omega}|\bigvee_{g\in F_{n}}F_{g,\omega}^{-1}\xi_{g\omega})(x)$, $\mathcal{I}$  is an \ $\sigma$ algebra made up of an invariant subset of \ $G$;

(2)$h^{(r)}_{\mu}(\mathbf{F},\xi)=\lim_{n\rightarrow\infty}\int H_{\mu_{\omega}}(\xi_{\omega}|\bigvee_{g\in F_{n}}F_{g,\omega}^{-1}\xi_{g\omega})dP(\omega)$;

(3)If\ $\mu$ is ergodic, then\ $\lim_{n\rightarrow\infty}\frac{1}{|F_{n}|}I_{\mu_{\omega}}(\bigvee_{g\in F_{n}}F_{g,\omega}^{-1}\xi_{g\omega})(x)=h_{\mu}^{(r)}(\mathbf{F},\xi),\mu-a.e.(\omega,x).$

\end{prop}

\begin{proof}

Set \ $g\in G$, $\mu\in \mathcal{P}_{\mathbb{P}}(\mathcal{E},G)$. From the action of \ $G$ on \ $\Omega$  and \ \cite{B}, it can be seen that for \ $\mathbb{P}-a.e.\ \omega\in\Omega$, there is
\begin{align}\label{w2}
F_{g,\omega}\mu_{\omega}=\mu_{g\omega},
\end{align}
for any \ $n\in \mathbb{N}$, denote \ $i_{n}=|F_{n}|$. Let \ $F_{n}=\{g_{1},g_{2},\cdot\cdot\cdot,g_{i_{n}}\}$, suppose $$r_{i_{n}-j}(\omega,x)=I_{\mu_{\omega}}\left(\xi_{\omega}|\bigvee^{i_{n}}_{m=j+1}F_{g_{j},g_{j}^{-1}\omega}\circ F_{g_{m}g_{j}^{-1}\omega}\xi_{g_{m}g_{j}^{-1}\omega}\right)(x).$$

By reusing the equation\ (2.1) and the properties of the information function
$$I_{\mu_{\omega}}\left(\alpha\bigvee\beta\right)=I_{\mu_{\omega}}(\alpha)+I_{\mu_{\omega}}(\beta|\alpha)$$
where \ $\alpha, \beta$ is any two finite partitions of \ $X$, then there is

\begin{equation*}
\begin{split}
\   I_{\mu_{\omega}}(\bigvee_{g\in F_{n}}F_{g,\omega}^{-1}\xi_{g\omega})(x) &=I_{\mu_{\omega}}\left(F_{g_{1},\omega}^{-1}\xi_{g_{1}\omega}\bigvee(\bigvee_{g\in F_{n}\backslash \{g_{1}\}}F_{g,\omega}^{-1}\xi_{g\omega})\right)(x)\\
&=I_{\mu_{\omega}}\left(F_{g_{1},\omega}^{-1}(\xi_{g_{1}\omega}\bigvee(\bigvee_{g\in F_{n}\backslash \{g_{1}\}}F_{g_{1},\omega}F_{g,\omega}^{-1}\xi_{g\omega}))\right)(x)\\
&=I_{\mu_{g_{1}\omega}}\left(\xi_{g_{1}\omega}\bigvee(\bigvee_{g\in F_{n}\backslash \{g_{1}\}}F_{g_{1},\omega}F_{g,\omega}^{-1}\xi_{g\omega})\right)(F_{g_{1},\omega}x)\\
&=I_{\mu_{g_{1}\omega}}\left(\xi_{g_{1}\omega}\big|\bigvee_{g\in F_{n}\backslash \{g_{1}\}}F_{g_{1},\omega}F_{g,\omega}^{-1}\xi_{g\omega}\right)(F_{g_{1},\omega}x)\\
&\ +I_{\mu_{g_{1}\omega}}(\bigvee_{g\in F_{n}\backslash \{g_{1}\}}F_{g_{1},\omega}F_{g,\omega}^{-1}\xi_{g\omega})(F_{g_{1},\omega}x)\\
&=r_{i_{n}-1}\circ\Theta_{1}(\omega,x)+I_{\mu_{\omega}}\left(\bigvee_{g\in F_{n}\backslash \{g_{1}\}}F_{g_{1},\omega}^{-1}\xi_{g\omega}\right)(x)\\
&=r_{i_{n}-1}\circ\Theta_{1}(\omega,x)+I_{\mu_{\omega}}\left(F_{g_{2},\omega}^{-1}\xi_{g_{2}\omega}\bigvee(\bigvee_{g\in F_{n}\backslash \{g_{1},g_{2}\}}F_{g,\omega}^{-1}\xi_{g\omega})\right)(x)\\
&=r_{i_{n}-1}\circ\Theta_{1}(\omega,x)+I_{\mu_{g_{2}\omega}}\left(\xi_{g_{2}\omega}\bigvee(\bigvee_{g\in F_{n}\backslash \{g_{1},g_{2}\}}F_{g_{2},\omega}F_{g_{1},\omega}^{-1}\xi_{g\omega})\right)(F_{g_{2},\omega}x)\\
&=r_{i_{n}-1}\circ\Theta_{1}(\omega,x)+I_{\mu_{g_{2}\omega}}\left(\xi_{g_{2}\omega}|\bigvee_{g\in F_{n}\backslash \{g_{1},g_{2}\}}F_{g_{2},\omega}F_{g_{1},\omega}^{-1}\xi_{g\omega}\right)(F_{g_{2},\omega}x)\\
&+I_{\mu_{\omega}}\left(\bigvee_{g\in F_{n}\backslash \{g_{1},g_{2}\}}F_{g,\omega}^{-1}\xi_{g\omega}\right)(x)\\
&=r_{i_{n}-1}\circ\Theta_{1}(\omega,x)+r_{i_{n}-2}\circ\Theta_{2}(\omega,x)+\cdot\cdot\cdot+r_{0}\circ\Theta_{i_{n}}(\omega,x)\\
&=\sum^{m=i_{n}}_{m=1}r_{i_{n}-m}\circ\Theta_{m}(\omega,x),
\end{split}
\end{equation*}
where\ $\Theta_{i_{j}}(\omega,x)=\Theta(g_{i_{j}},(\omega,x)), \forall j=1,2,...i_{n}, \ r_{0}\circ\Theta_{i_{n}}(\omega,x)=I_{\mu_{\omega}}(\xi_{\omega})(x)$. Then
$$\lim_{n\rightarrow\infty}r_{i_{n}}(\omega,x)=\lim_{n\rightarrow\infty}I_{\mu_{\omega}}(\xi_{\omega}|\bigvee^{i_{n}}_{m=1}F^{-1}_{g_{m},\omega}\xi_{g\omega})(x)=r(\omega,x),\mu-a.e.(\omega,x).$$
can be inferred from the martingale convergence theorem. In addition, we can get   \ $\int\sup_{i_{n}}r_{i_{n}}(\omega,x)d\mu(\omega,x)\leq\int H_{\mu_{\omega}}(\xi_{\omega})d\mathbb{P}(\omega)+1<\infty$ from\ Chung lemma\ \cite{P}, then \ $r\in L^{1}(\Omega\times X,\mu)$ is obtained by the control convergence theorem, thus, from the Theorem 4.28 in\ \cite{KL}, we can see that
$$\lim_{n\rightarrow\infty}\frac{1}{|F_{n}|}\sum_{g\in F_{n}}r\circ \Theta_{g}(\omega,x)=E_{\mu}(r|\mathcal{I})(\omega,x),\mu-a.e. \text{in} \ L^{1}(\Omega\times X,\mu)$$
that is
$$\lim_{n\rightarrow\infty}\frac{1}{|F_{n}|}I_{\mu_{\omega}}\left(\bigvee^{i_{n}}_{m=1}F^{-1}_{g,\omega}\xi_{g\omega}\right)(x)=E_{\mu}(r|I)(\omega,x),\mu-a.e.(\omega,x),$$
Thus\ (1) is proven to be true.

Secondly, we start to prove (2), since
\begin{equation*}
\begin{split}
h^{(r)}_{\mu}(\mathbf{F},\xi)&=\lim_{n\rightarrow\infty}\frac{1}{|F_{n}|}\int H_{\mu_{\omega}}(\bigvee_{g\in F_{n}}F_{g,\omega}^{-1}\xi_{g\omega})d\mathbb{P}(\omega)\\
&=\lim_{n\rightarrow\infty}\frac{1}{|F_{n}|}\int I_{\mu_{\omega}}(\bigvee_{g\in F_{n}}F_{g,\omega}^{-1}\xi_{g\omega})(x)d\mu\\
&=\int E_{\mu}(r|\mathcal{I})(\omega,x)d\mu\\
&=\int r(\omega,x)d\mu\\
&=\lim_{n\rightarrow\infty}\int H_{\mu_{\omega}}(\xi_{\omega}|\bigvee_{g\in F_{n}}F_{g,\omega}^{-1}\xi_{g\omega})d\mathbb{P}(\omega).
\end{split}
\end{equation*}

Finally, proving (3), a sequence \ $C(n)\subseteq \xi^{F_{n}}$ satisfying
$$|C(n)|\leq2^{(h^{(r)}_{\mu}(\mathbf{F},\xi)+\eta)|F_{n}|},$$
and for \ $\mu-a.e. (\omega,x)$, and sufficiently large \ $n$, there are \ $\xi^{F_{n}}_{\omega}(x)\in C(n)$ and
$$
|(C(n))_{\omega}|\leq|C(n)| \text{and}\ \xi_{\omega}^{F_{n}}(x)\in (C(n))_{\omega}, \ \mu-a.e. (\omega,x),
$$
can be found using Lemma \ \ref{lem11} and the Pointwise Ergodic Theorem.

Let
$$
M(n)=\left\{(\omega,x)\in \cup (C(n)):\frac{-\log\mu_{\omega}(\xi_{\omega}^{F_{n}}(x))}{|F_{n}|}>h_{\mu}^{(r)}(\mathbf{F},\xi)+2\eta\right\}.
$$

Since
$$\mu(M(n))\leq |C(n)|\times 2^{(h^{(r)}_{\mu}(\mathbf{F},\xi)+\eta)|F_{n}|}\leq 2^{-\eta|F_{n}|},$$
gets
$$\Sigma^{\infty}_{n=1}\mu(M(n))<\infty.$$
Using the Borel-Cantelli Lemma, for a sufficiently large\ $n$, $a.e.(\omega,x)$ is not in \ $M(n)$, and then
\begin{align}\label{91}
\limsup_{n\rightarrow\infty}\frac{-\log\mu_{\omega}(\xi^{F_{n}}_{\omega}(x))}{|F_{n}|}\leq h^{(r)}_{\mu}(\mathbf{F},\xi)+2\eta,\ \mu-a.e. (\omega,x).
\end{align}
is obtained.

On the other hand, using the lemma\ \ref{lem12}, if a set of regular measures for $(\omega,x)$ satisfies
\begin{align}\label{92}
\liminf_{n\rightarrow\infty}\frac{-\log\mu_{\omega} (\xi^{F_{n}}_{\omega}(x))}{|F_{n}|}< h^{(r)}_{\mu}(\mathbf{F},\xi)-2\eta,
\end{align}
then for any  $\delta>0$, if \ $n$ is large enough, then there is a set \ $\mathcal{E}'$ of measures \ $1-\delta$ such that for any \ $(\omega,x)$ there is \ $M(x)=\{F_{n_{i}}a\}_{i\in \mathcal{I}}$ satisfying the following conditions

1) $n_{i}\geq\delta^{-1}$ and for any\ $i\in \mathcal{I}, F_{n_{i}}a_{i}\subset F_{n}$;

2) $\mu(\xi^{F_{n}}(a_{i}(\omega,x)))>2^{(h^{(r)}_{\mu}(\mathbf{F},\xi)-2\eta)|F_{n_{i}}|}$;

3) $(1+\delta)|\cup F_{n_{i}}a_{i}|\geq \sum_{i\in \mathcal{I}}|F_{n_{i}}|\geq(1-\delta)|F_{n}|$.

As long as \ $\delta$ is sufficiently small, then the number of elements covering \ $\mathcal{E}'$ in \ $\xi^{F_{n}}$  does not exceed \ $2^{(h^{(r)}_{\mu}(\mathbf{F},\xi)-\eta)|F_{n}|}$, so that the following relation can be defined by \ $h^{(r)}_{\mu}(\mathbf{F},\xi)$

\begin{equation*}
\begin{split}
h^{(r)}_{\mu}(\mathbf{F},\xi)&=\lim_{n\rightarrow\infty}\frac{1}{|F_{n}|}H_{\mu}(\xi_{F_{n}}|\mathcal{F}_{\mathcal{E}})\\
&\leq\lim_{n\rightarrow\infty}\frac{1}{|F_{n}|}\log|\xi^{F_{n}}|\\
&\leq\lim_{n\rightarrow\infty}\frac{1}{|F_{n}|}\log2^{(h^{(r)}_{\mu}(\mathbf{F},\xi)-\eta)|F_{n}|}\\
&<h^{(r)}_{\mu}(\mathbf{F},\xi)-\eta.
\end{split}
\end{equation*}
Contradictions, therefore
\begin{align}\label{93}
\liminf_{n\rightarrow\infty}\frac{-\log\mu_{\omega}(\xi_{\omega}^{F_{n}}(x))}{|F_{n}|}\geq h^{(r)}_{\mu}(\mathbf{F},\xi)-2\eta.
\end{align}
Thus \ (3) is established.
\end{proof}

{\bf Author contributions} All authors contributed to the study conception and design. The first draft of the manuscript was written by Yuan Lian, both the second author and the first author analysed and verified the manuscript. and all authors commented on previous versions of the manuscript. All authors read and approved the final manuscript.

{\bf Data Availability Statement} The authors declare that this research does not use data from any external sources or counterparts.

{\bf Declarations}

{\bf Competing Interest}  The authors have no relevant financial or non-financial interests to disclose.

\bigskip \noindent{\bf References}



\end{document}